\documentclass[aop,preprint]{imsart}

\RequirePackage[OT1]{fontenc}
\RequirePackage{amsthm,amsmath, amssymb}
\RequirePackage[numbers]{natbib}
\RequirePackage[colorlinks,citecolor=blue,urlcolor=blue]{hyperref}
\usepackage{graphicx}

\arxiv{1412.1896}

\startlocaldefs
\numberwithin{equation}{section}
\theoremstyle{plain}
\newtheorem{theorem}{Theorem}[section]
\newtheorem{lemma}{Lemma}[section]
\newtheorem{proposition}{Proposition}[section]
\newtheorem{corollary}{Corollary}[section]
\newtheorem{remark}{Remark}[section]
\endlocaldefs

\begin{document}

\begin{frontmatter}
\title{On structure of regular Dirichlet subspaces for one-dimensional Brownian motion}
\runtitle{On structure of regular Dirichlet subspaces}
\thankstext{T1}{{The first named author is partially supported by a joint grant (No. Y690021G22) of CPSF and CAS. The second named author is partially supported by NSFC No. 11271240.}}

\begin{aug}
\author{\fnms{Liping} \snm{Li} \thanksref{t1}
\ead[label=e1]{lipingli10@fudan.edu.cn}}
\and \author{\fnms{Jiangang} \snm{Ying}
\ead[label=e2]{jgying@fudan.edu.cn}}
\ead[label=u1,url]{http://homepage.fudan.edu.cn/jgying}

\thankstext{t1}{Corresponding author.}
\runauthor{Li L. and Ying J.}

\affiliation{Fudan University 
}

\address{220 Handan Road, Fudan University
              \\
Shanghai, China. 200433.\\
\printead{e1}\\
\phantom{E-mail:\ }\printead*{e2}}

\end{aug}

\begin{abstract}
The main purpose of this paper is to explore the structure of regular Dirichlet subspaces of 1-dim Brownian motion.
As stated in \cite{FMG} every such regular Dirichlet subspace can be characterized by a measure-dense set $G$.
When $G$ is open, $F=G^c$ is the
boundary of $G$ and, before leaving $G$, the diffusion associated with the regular Dirichlet subspace is nothing but Brownian motion. Their traces on $F$ still inherit the inclusion relation,
in other words, the trace Dirichlet form of regular Dirichlet subspace on $F$ is still a regular Dirichlet subspace of trace Dirichlet form of one-dimensional Brownian motion on $F$.
Moreover we shall prove that the trace of Brownian motion on $F$ may be decomposed into two part, one is the trace of the regular Dirichlet subspace on $F$, which has only
the non-local part and the other comes from the orthogonal complement of the regular Dirichlet subspace, which has only the local part. Actually
the orthogonal complement of regular Dirichlet subspace corresponds to {a time-changed absorbing Brownian motion} after a darning transform.
\end{abstract}



\end{frontmatter}

\section{Introduction}\label{INTRO}
A Dirichlet form is an analytic characterization of a symmetric Markov process and its structure
means roughly its Beurling-Deny decomposition with its domain.
The main purpose of the current paper is to describe the structure of a regular Dirichlet subspace
and its orthogonal complement for 1-dimensional Brownian motion
using `trace' method developed by \cite{CFJ}, through which, we may more easily observe that
a Dirichlet subspace differs from the original form
not only in their domains but also in Beurling-Deny decompositions, which give more information
about the structure.

Roughly speaking a regular Dirichlet subspace of a Dirichlet form is a subspace
which is also a regular Dirichlet form on the same state space. Since a regular
Dirichlet subspace corresponds to a unique symmetric Markov process, 
the problem of existence and characterization of regular Dirichlet subspaces of a Dirichlet form 
is a basic and interesting problem for us to explore 
in the theory of Dirichlet forms.
This problem was raised and discussed by the second author and his co-authors in \cite{FMG}, in which 
a complete characterization for regular Dirichlet subspaces of one-dimensional Brownian motion was given.
In fact such a regular Dirichlet subspace may be characterized by a measure-dense set $G$,
by which we mean that $G\cap(a,b)$ has positive measure for any non-empty interval $(a,b)\subset\mathbb{R}$
and the Markov process associated with the regular Dirichlet subspace is the one-dimensional diffusion
with scaling function $$s(x)=\int_0^x 1_G(y)dy$$ and speed measure being Lebesgue measure. 
By the way, {a complete characterization of regular Dirichlet subspaces even} for multi-dimensional Brownian motions has not been obtained
due to lack of something like scaling functions in 1-dimensional case. 

With the characterization at hand, it is natural to be 
curious about the structure of regular Dirichlet subspaces. 
It is not intuitive how a regular Dirichlet subspace differs and how the corresponding process
moves because, by observing the forms, the subspace seems to be only different from the original form
in their domains. 
However in the case of 1-dim Brownian motion, we may see from its scaling function that the process corresponding 
to the regular Dirichlet subspace moves like Brownian motion
on $G$ more or less but spends almost no time on $G^c$ though it runs all over $G^c$ which has positive measure.
We aim to find a more intuitive picture to describe 
the precise structure of regular Dirichlet subspaces of one-dimensional Brownian motion.

As mentioned, the approach we use to explore the structure 
is the method of trace. Generally a Dirichlet form may be decomposed
into a minimal process on an open subset $G$ and its orthogonal complement, which is called the trace of Dirichlet
form on $G^c$. Usually the trace is the Dirichlet form corresponding to the process obtained by the original
process through a time change induced by a positive continuous additive functional. How to describe the trace of a form dates back to J. Douglas who gave a complete
characterization of the trace of the form, associated to Brownian motion living on closed unit disc, on its boundary in \cite{Douglas}. The similar
characterization has been done for general symmetric Dirichlet form by the second author and his co-authors in \cite{CM} and \cite{CFJ}.
In the current article, we shall prove that
when the measure-dense $G$ is open and $F=G^c$ has positive measure, the trace $\check{X}$ of 1-dim Brownian motion $X$ on $F=G^c$
is non-trivial and (its Beurling-Deny decomposition) has both diffusion part and jump part as expected,  the trace $\check X^{(s)}$ of the regular Dirichlet subspace $X^{(s)}$ of $X$ is a regular Dirichlet subspace
of the trace Brownian motion
$\check{X}$ which has only the jump part, and finally the remaining part is the orthogonal complement of the regular Dirichlet subspace
whose Beurling-Deny decomposition has only the diffusion part.
In addition, we show that the regular representation of the orthogonal complement is the darning transform of Brownian motion.
From this result we can see that though the process $X^{(s)}$ corresponding to the regular Dirichlet subspace moves continuously on $F$ but
it looks like jumping due to the special structure of $F$. In other words, it `flies like the wind and leaves no shadow'.

Let $E$ be a locally compact separable metric space and $\xi$ a Radon measure fully supported on $E$. We refer the terminologies of Dirichlet forms on the Hilbert space $L^2(E,\xi)$ to \cite{CM} and \cite{FU}. Assume that $(\mathcal{E}^1,\mathcal{F}^1)$ and $(\mathcal{E}^2,\mathcal{F}^2)$ are two regular Dirichlet forms on $L^2(E,\xi)$. Then $(\mathcal{E}^1,\mathcal{F}^1)$ is called a \emph{regular Dirichlet subspace} of $(\mathcal{E}^2,\mathcal{F}^2)$ 
if
\begin{equation}
	\mathcal{F}^1\subset \mathcal{F}^2,\quad \mathcal{E}^2(u,v)=\mathcal{E}^1(u,v),\quad u,v\in \mathcal{F}^1.
\end{equation}
If in addition $\mathcal{F}^1$ is a proper subset of $\mathcal{F}^2$, then we say $(\mathcal{E}^1,\mathcal{F}^1)$ is a \emph{proper regular Dirichlet subspace} of $(\mathcal{E}^2,\mathcal{F}^2)$.

We denote the Lebesgue measure on $\mathbb{R}$ by $m$. It is well known that the 1-dimensional Brownian motion is symmetric with respect to $m$ and its associated Dirichlet form on $L^2(\mathbb{R})$ is $(\mathcal{E,F}):=(\frac{1}{2}\mathbf{D},H^1(\mathbb{R}))$, where $H^1(\mathbb{R})$ is the 1-Sobolev space and for any $u,v\in H^1(\mathbb{R})$,
\[
	\mathbf{D}(u,v)=\int_\mathbb{R} u'(x)v'(x)dx.
\] As stated in \cite{FMG} and \cite{XPJ}, fix a strictly increasing and absolutely continuous function $s$ on $\mathbb{R}$ satisfying
\begin{equation}\label{EQSX}
	s'(x)=0 \text{ or } 1 \quad\text{a.e.}
\end{equation}
 and define a symmetric bilinear form $(\mathcal{E}^{(s)},\mathcal{F}^{(s)})$ on $L^2(\mathbb{R})$ by
\begin{equation}\label{SBM}
 \begin{aligned}
       & \mathcal{F}^{(s)}:= \{u\in L^2(\mathbb{R}):\;u\ll s,\;\int_{\mathbb{R}}\left(\frac{du}{ds}\right)^2ds<\infty\}, \\
      &  \mathcal{E}^{(s)}(u,v):= \frac{1}{2}\int_{\mathbb{R}} \frac{du}{ds}\frac{dv}{ds}ds,\quad u,v\in \mathcal{F}^{(s)},
        \end{aligned}
    \end{equation}
  where $u\ll s$ means $u$ is absolutely continuous with respect to $s$.
Then $(\mathcal{E}^{(s)},\mathcal{F}^{(s)})$ is a regular Dirichlet subspace of $(\mathcal{E,F})$ and $s$ is called the scaling function of $(\mathcal{E}^{(s)},\mathcal{F}^{(s)})$. The associated diffusion of $(\mathcal{E}^{(s)},\mathcal{F}^{(s)})$ is denoted by $X^{(s)}$. {Denote by $\mathcal{F}_\text{e}$ and $\mathcal{F}^{(s)}_\text{e}$} the extended Dirichlet spaces of $(\mathcal{E,F})$ and $(\mathcal{E}^{(s)},\mathcal{F}^{(s)})$ respectively. Note that $$\mathcal{F}_\text{e}=H^1_\text{e}(\mathbb{R}):=\{u: u\text{ is absolutely continuous on }\mathbb{R} \text{ and }u'\in L^2(\mathbb{R})\}.$$
{To the contrary}, if $(\mathcal{E}',\mathcal{F}')$ is a regular Dirichlet subspace of $(\mathcal{E,F})$, there always exists a strictly increasing and absolutely continuous function $s$ on $\mathbb{R}$ satisfying \eqref{EQSX} such that $(\mathcal{E}',\mathcal{F}')=(\mathcal{E}^{(s)},\mathcal{F}^{(s)})$.  To see this we refer the recurrent case to Theorem~2.1 of \cite{FMG}. Generally since $(\mathcal{E,F})$ is strongly local and irreducible, it follows from Theorem~4.6.4 of \cite{FU} and Theorem~1 of \cite{LY} that $(\mathcal{E}',\mathcal{F}')$ is also strongly local and irreducible. The irreducibility of $(\mathcal{E}',\mathcal{F}')$ implies that
\[
	P'_x(\sigma_y<\infty)>0,\quad x,y\in \mathbb{R},
\]
where $(P'_x)_{x\in\mathbb{R}}$ is the class of probability measures of associated diffusion $X'$ of  $(\mathcal{E}',\mathcal{F}')$ and $\sigma_y$ is the hitting time of $\{y\}$ relative to $X'$, see Theorem~4.7.1 of \cite{FU}. Then from \cite{XPJ} we can deduce that $X'$ can be characterized by a scaling function $s$ and symmetric measure $m$. In particular it corresponds to Dirichlet form \eqref{SBM}.

Fix a regular Dirichlet subspace $(\mathcal{E}^{(s)},\mathcal{F}^{(s)})$ of $(\mathcal{E,F})$ and its scaling function $s$. Let $$G:=\{x\in \mathbb{R}: s'(x)=1\}.$$ Then $G$ is defined in the sense of almost everywhere and it holds that
\begin{equation}\label{EQMGC}
	 m(G\cap (a,b))>0, \quad \forall (a,b)\subset \mathbb{R}.
	 \end{equation}
Note that the condition \eqref{EQMGC} of $G$ is equivalent to that $s$ is strictly increasing.
 In particular $(\mathcal{E}^{(s)},\mathcal{F}^{(s)})$ is a proper regular Dirichlet subspace of $(\frac{1}{2}\mathbf{D},H^1(\mathbb{R}))$ if and only if the Lebesgue measure of $F:= G^c$ is positive, i.e. $m(F)>0$.  On the other side if we have a subset $G$ of $\mathbb{R}$ such that \eqref{EQMGC} holds for any open interval $(a,b)$, then
\begin{equation}\label{EQSFX}
	ds:=1_G(x)dx
\end{equation}
defines a class of scaling functions satisfying \eqref{EQSX}, whereas they only differ up to a constant. In other words, the subset $G$ satisfying \eqref{EQMGC} is one-to-one corresponding to the scaling function $s$ with condition \eqref{EQSX} up to a constant.
Thus $G$ is an essential characteristic of  $(\mathcal{E}^{(s)},\mathcal{F}^{(s)})$.

In this paper we shall always make the following assumption on $G$:
\begin{description}
\item[(\textbf{H})] $G$ is an open set satisfying \eqref{EQMGC} and $m(F)>0$.
\end{description}
In fact the typical example of $F$ is a generalized Cantor set which is actually closed. Hence this assumption is very natural. But we still
want to point out that this assumption is not trivial. An example of a set $G$ satisfying \eqref{EQMGC} but having no open version can be constructed as follows.  Take a set $J\subset \mathbb{R}$ such that for any finite open interval $I$, it holds that
\[
	0<m(J\cap I)<m(I).
\]
We refer the existence of $J$ to \S1.5 of \cite{FGB}. But $J$ has no open a.e. version. To this end, assume that $G$ is an open a.e. version of $J$, i.e. $m(J\Delta G)=0$ and $G$ is open. Take a finite open interval $I\subset G$ and it follows that $m(I)=m(I\cap G)=m(I\cap J)<m(I)$ which conduces to a contradiction.

When $G$ satisfies (\textbf{H}), we may always assume without loss of generality that $F$ has no isolated points.
In fact 
let $$f(x):=\int_0^x 1_{F}(y)dy$$
and $$\tilde{G}:=\{x: \text{there exists}\ \delta>0\ \text{such that}\ f(x-\delta)=f(x+\delta)\}.$$
Then $\tilde {G}$ is an open a.e. version of $G$ and the complement of it has no isolated points.
In the sequel, we shall impose this assumption.

 Since $G$ is open, we can write
\begin{equation}\label{EQGCN}
	G=\bigcup_{n=1}^\infty I_n,
\end{equation}
where $\{I_n=(a_n,b_n):n\geq 1\}$ is a series of disjoint open intervals. Clearly at most two of them are infinite. Denote all finite endpoints of $\{I_n:n\geq 1\}$ by
\begin{equation}\label{EQHAB}
	H:=\{a_n,b_n:n\geq 1\}\setminus \{-\infty,\infty\}
\end{equation}
and let $d_n:=|b_n-a_n|$ for any $n\geq 1$. Note that $H\subset F$ and any point in $F\setminus H$ is a limitation of a subsequence of $H$. Clearly any two different intervals
 $I_n$ and $I_m$ can not share a common endpoint due to our assumption that $F$ has no isolated points.

The structure of this paper is as follows. In \S\ref{TBM} we shall first prove that before leaving $G$, the diffusion $X^{(s)}$ is equivalent to one-dimensional Brownian motion, see Lemma~\ref{THM1}. Then as stated in Theorem~\ref{THM5} we find that the trace Dirichlet form of $(\mathcal{E}^{(s)},\mathcal{F}^{(s)})$ on $F$ is a regular Dirichlet subspace of trace Dirichlet form of $(\mathcal{E,F})$ on $F$. Moreover the former Dirichlet form is a non-local Dirichlet form whereas the latter one is a mix-type Dirichlet form. Their common jumping measure $U$ is supported on countable points in $F\times F\setminus d$:
\[
	\{(a_n,b_n), (b_n,a_n):a_n>-\infty,b_n<\infty, n\geq 1\},
\]
where $a_n,b_n$ are endpoints of $I_n$ in \eqref{EQGCN}. In particular,
\[
	U((a_n,b_n))=U((b_n,a_n))=\frac{1}{2d_n}.
\]	
Thus we write the precise expressions of these two trace Dirichlet forms in Theorem~\ref{THM5}.

Since the smaller trace Dirichlet form only inherits the non-local part of bigger one, our concern in \S\ref{OCO} is whether and how we can describe the remaining information, i.e. the strongly local part, of trace Dirichlet form of one-dimensional Brownian motion on $F$. In order to do that, we first characterize the orthogonal complement of regular Dirichlet subspace. Although $\mathcal{F}_\text{e}$ is not a Hilbert space relative to the quadratic form $\mathcal{E}$, we can still define the orthogonal complement $\mathcal{G}^{(s)}$ of $\mathcal{F}^{(s)}_\text{e}$ in $\mathcal{F}_\text{e}$ relative to $\mathcal{E}$ in form, that is
$$\mathcal{G}^{(s)}:=\{u\in \mathcal{F}_\text{e}:\mathcal{E}(u,v)=0\text{ for any }v\in \mathcal{F}^{(s)}_\text{e}\}.$$
In Theorem~\ref{THM3} we shall describe the decomposition of any $u\in \mathcal{F}_\text{e}$ related to $\mathcal{F}^{(s)}_\text{e}$ and $\mathcal{G}^{(s)}$. In particular if $m(G)=\infty$, then $(\mathcal{E},\mathcal{G}^{(s)})$ is a Dirichlet space in wide sense, i.e. its satisfies all conditions of Dirichlet form except for the denseness in {$L^2(I)$}, see Lemma~\ref{LM6}. A Dirichlet form in wide sense is also called a D-space in the terminologies of \cite{FU2}. By a darning transform which regards each component $I_n$ of $G$ as a new point, $\mathcal{G}^{(s)}$ can actually be realized as a regular strongly local Dirichlet form and this regular Dirichlet form is a regular representation of $\mathcal{G}^{(s)}$ in the context of \cite{FU2}, see Theorem~\ref{prop5}. Similarly we can define the orthogonal complement of trace Dirichlet space, say \eqref{EQMGSUF}. This orthogonal complement exactly inherits the strongly local part of trace Brownian motion on $F$, which will be stated in Theorem~\ref{THM5}. Moreover a similar darning transform makes it be a regular strongly local Dirichlet form which is also equivalent to the darning transform of $\mathcal{G}^{(s)}$. Their associated diffusion is called the orthogonal darning process which is actually {a time-changed absorbing Brownian motion}, see {Theorems~\ref{prop5} and \ref{THM6}}.

\section{Traces of Brownian motion and their regular Dirichlet subspaces}\label{TBM}

We first prove a useful lemma.

\begin{lemma}\label{LM3}
	$\mathcal{F}^{(s)}=\{u\in H^1(\mathbb{R}): u'=0\text{ a.e. on }F\}$.
\end{lemma}
\begin{proof}
	For any $u\in \mathcal{F}^{(s)}$, there exists an absolutely {continuous} function $\phi$ such that $u(x)=\phi(s(x))$. Then
	\[
		u'(x)=\phi'(s(x))\cdot s'(x)=\phi'(s(x))1_G(x)
	\]
and hence $u'=0$ a.e. on $F$. {To the contrary}, let $u\in H^1(\mathbb{R})$ and $u'=0$ a.e. on $F$. Then
\[
	u(x)-u(0)=\int_0^x u'(y)dy=\int_0^xu'(y)1_G(y)dy=\int_0^xu'(y)ds(y).
\]
Thus $u$ is absolutely continuous with respect to $s$ and $du/ds=u'(x)$, $ds$-a.e. It follows from $u'\in L^2(\mathbb{R})\subset L^2(\mathbb{R},ds)$ that $u\in \mathcal{F}^{(s)}$.  
\end{proof}

From the above lemma we can deduce a simple but very interesting property of regular Dirichlet subspace $(\mathcal{E}^{(s)},\mathcal{F}^{(s)})$. We first give some notes about the part Dirichlet forms. The part Dirichlet form of $(\mathcal{E}^{(s)},\mathcal{F}^{(s)})$ on $G$, denoted by $(\mathcal{E}^{(s)}_G,\mathcal{F}^{(s)}_G)$, is defined by
\[
	\begin{aligned}
		&\mathcal{F}^{(s)}_G:=\{u\in \mathcal{F}^{(s)}: u(x)=0,\; x\in F\},\\
		&\mathcal{E}^{(s)}_G(u,v):=\mathcal{E}^{(s)}(u,v),\quad u,v\in \mathcal{F}^{(s)}_G.
	\end{aligned}
\]
It is regular on $L^2(G)$ and corresponds to the Markov process $(X^{(s)}_t)_{t<\tau^{(s)}_{G}}$ on $G$ with the life time $\tau^{(s)}_G$,  where $\tau^{(s)}_G$ is the first exist time of $G$ relative to $X^{(s)}$. Similarly we can write $(\mathcal{E}_G,\mathcal{F}_G)$ or $(\frac{1}{2}\mathbf{D}_G,H^1_0(G))$ for the part Dirichlet form of $(\mathcal{E,F})$ on $G$. The following lemma indicates that before leaving $G$, the process $X^{(s)}$ is equivalent to  one-dimensional Brownian motion.

\begin{lemma}\label{THM1}
	It holds that $(\mathcal{E}^{(s)}_G,\mathcal{F}^{(s)}_G)= (\frac{1}{2}\mathbf{D}_G,H^1_0(G))$.
\end{lemma}
\begin{proof}
	Clearly $\mathcal{F}^{(s)}_G\subset H^1_0(G)$ and for any $u,v\in \mathcal{F}^{(s)}_G$,  $\mathcal{E}_G(u,v)=\mathcal{E}^{(s)}_G(u,v)$. Thus it suffices to prove that $\mathcal{F}^{(s)}_G=H^1_0(G)$. Fix $u\in H^1_0(G)$. Since $u$ is absolutely continuous, it is a.e. differentiable. Thus for a.e. $x\in F$ at where $u$ is differentiable, take a sequence $\{x_n:n\geq 1\}\subset F$ which is convergent to $x$ as $n\rightarrow \infty$. Note that $u=0$ on $F$. Then we have
	\[
		u'(x)=\lim_{n\rightarrow \infty}\frac{u(x_n)-u(x)}{x_n-x}=0.
	\]
Hence it follows from Lemma~\ref{LM3} that $u\in \mathcal{F}^{(s)}$ whereas $u=0$ on $F$. Therefore $u\in \mathcal{F}^{(s)}_G$.  
\end{proof}

Recall that the scaling function $s$ of $X^{(s)}$ satisfies that $s'=1$ a.e. on $G$. That means $X^{(s)}$ has the same scale (up to a constant) as one-dimensional Brownian motion on $I_n$ for any $n\geq 1$, where $\cup_{n\geq 1}I_n=G$.  From this aspect we can see that the above theorem is natural and reasonable.

Set
\[
	\mathcal{F}_{\text{e},G}=H^1_\text{e}(G):=\{u\in \mathcal{F}_\mathrm{e}:u=0\text{ on }F\}
\]
and
\[
	\mathcal{F}^{(s)}_{\text{e},G}:=\{u\in \mathcal{F}^{(s)}_\text{e}:u=0\text{ on }F\}.
\]
Note that if $s(-\infty)>\infty$ (resp. $s(\infty)<\infty$) then $F$ is not bounded below (resp. above), in other words, there exists a sequence $\{x_n\}\subset F$ such that $x_n\rightarrow -\infty$ (resp. $x_n\rightarrow \infty$). Hence if $u\in \mathcal{F}_\text{e}$ such that $u=0$ on $F$, it follows that $\lim_{x\rightarrow -\infty}u(x)=0$ (resp. $\lim_{x\rightarrow \infty}u(x)=0$). As a consequence we have the following result.

\begin{lemma}\label{COR3}
	It holds that $\mathcal{F}_{\text{e},G}=\mathcal{F}^{(s)}_{\text{e},G}$.
\end{lemma}

Set further
\[
	\mathcal{H}_F:=\{u\in \mathcal{F}_\text{e}: \mathcal{E}(u,w)=0 \text{ for any }w\in \mathcal{F}_{\text{e}, G}\}
\]
and
\[
	\mathcal{H}^{(s)}_F:=\{u\in \mathcal{F}^{(s)}_\text{e}: \mathcal{E}^{(s)}(u,w)=0 \text{ for any }w\in \mathcal{F}^{(s)}_{\text{e},G}\}.
\]
Then every $u\in \mathcal{F}_\text{e}$ can be expressed uniquely as (see Exercise~4.6.4 of  \cite{FU})
\[
	u=u_1+u_2,\quad u_1\in \mathcal{F}_{\text{e},G},u_2\in \mathcal{H}_F.
\]
We denote the $\mathcal{H}_F$-part $u_2$ of $u$ by $H_Fu$. Similarly every $v\in \mathcal{F}^{(s)}_\text{e}$ can be expressed uniquely as
\begin{equation}\label{EQVVV}
	v=v_1+v_2,\quad v_1\in \mathcal{F}^{(s)}_{\text{e},G},v_2\in \mathcal{H}^{(s)}_F.
\end{equation}
Denote the $\mathcal{H}^{(s)}_F$-part $v_2$ of $v$ by $H^{(s)}_Fv$. Note that if $(\mathcal{E}^{(s)},\mathcal{F}^{(s)})$ is transient, then $\mathcal{H}^{(s)}_F$ is the orthogonal complement of $\mathcal{F}^{(s)}_{\text{e},G}$ with respect to the inner product $\mathcal{E}^{(s)}$, i.e.
\[
	\mathcal{F}^{(s)}_\text{e}=\mathcal{F}^{(s)}_{\text{e},G}\oplus_{\mathcal{E}^{(s)}} \mathcal{H}^{(s)}_F.
\]

We now turn to trace Dirichlet forms.
 Let $X=(X_t:t\geq 0)$ be the one-dimensional Brownian motion on $\mathbb{R}$ corresponding to $(\mathcal{E,F})$. As stated in Lemma~\ref{THM1} the part Dirichlet form of $(\mathcal{E}^{(s)},\mathcal{F}^{(s)})$ on $G$ is the same as the part of $(\mathcal{E,F})$ on $G$. That means that before leaving $G$, $X^{(s)}$ is equivalent to $X$. Since $(\mathcal{E}^{(s)},\mathcal{F}^{(s)})$ is a proper regular Dirichlet subspace of $(\mathcal{E,F})$, we guess that their trace Dirichlet forms on the boundary $F$ may inherit the inclusion relation between $(\mathcal{E,F})$ and $(\mathcal{E}^{(s)},\mathcal{F}^{(s)})$.

Let $\mu$ be a Radon (smooth) measure on $F$. A set $K$ is called the \emph{support} of $\mu$ if $K$ is the smallest closed set outside of which $\mu$ vanishes. We refer the definition of the \emph{quasi-support} of $\mu$ (relative to $(\mathcal{E,F})$ or $(\mathcal{E}^{(s)},\mathcal{F}^{(s)})$) to  \cite{CM}. Note that an $\mathcal{E}^{(s)}$-quasi-continuous function is always $\mathcal{E}$-quasi-continuous, hence it is continuous. It follows that an ($\mathcal{E}^{(s)}$ or $\mathcal{E}$)-quasi-closed set is always closed. Hence we know that the support of $\mu$ is also the quasi-support of $\mu$. In this section we always assume that the support of $\mu$ is $F$. The following lemma indicates that $1_F(x)dx$ is an example of such a measure $\mu$ on $\mathbb{R}$.

\begin{lemma}
Assume that $\mu(dx)=1_F(x)dx$. Then $\mu$ is a Radon smooth measure with respect to $X$ and $X^{(s)}$. Moreover the support and quasi-support of $\mu$ are both $F$.
\end{lemma}
\begin{proof}
Clearly $\mu$ is Radon.	Since the $m$-polar set of $X$ and $X^{(s)}$ must be empty set, it follows that $\mu$ is  smooth with respect to $X$ and $X^{(s)}$. Let $K$ be the support of $\mu$. Then $K\subset F$. If $K\neq F$, take $x\in F\setminus K$. Since $K$ is closed, we have
	\[
		d(x,K)=\inf_{y\in K}|x-y|>0.
	\]
Fix a constant $\epsilon<d(x,K)/2$. Let $H_\epsilon:=F\cap (x-\epsilon,x+\epsilon)$. Clearly $H_\epsilon\subset F\setminus K$ and $m(H_\epsilon)>0$. Thus $\mu(K^c)=m(F\setminus K)>m(H_\epsilon)>0$ which conduces to a contradiction. Therefore $K=F$. 
\end{proof}

 Denote the time-changed processes of $X$ and $X^{(s)}$ with respect to $\mu$ by $\check{X}$ and $\check{X}^{(s)}$ respectively. Then $\check{X}$ and $\check{X}^{(s)}$ are both $\mu$-symmetric on $F$ and their corresponding Dirichlet forms are both regular on $L^2(F,\mu)$. Denote these two associated Dirichlet forms, i.e. the traces of $(\mathcal{E,F})$ and $(\mathcal{E}^{(s)},\mathcal{F}^{(s)})$ on $F$, by $(\check{\mathcal{E}},\check{\mathcal{F}})$ and $(\check{\mathcal{E}}^{(s)},\check{\mathcal{F}}^{(s)})$, respectively. 
Precisely, let $\sigma_F$ and $\sigma^{(s)}_F$ be the hitting time of $F$ relative to $K$ and $X^{(s)}$ and in fact we have
\[
	H_Fu(x)=E^xu(X_{\sigma_F}),\quad x\in \mathbb{R}
\]
for any $u\in \mathcal{F}_\text{e}$ and
\[
	H^{(s)}_Fu(x)=E^x u(X^{(s)}_{\sigma^{(s)}_F}),\quad x\in \mathbb{R}
\]
for any $u\in \mathcal{F}^{(s)}_\text{e}$. Then
\[
\begin{aligned}
	&\check{\mathcal{F}}=\{\varphi\in L^2(F,\mu): \varphi=u\; \mu\text{-a.e. on }F\text{ for some }u\in \mathcal{F}_\mathrm{e}\},	\\
	&\check{\mathcal{E}}(\varphi,\varphi)=\mathcal{E}(H_Fu,H_Fu),\quad \varphi\in \check{\mathcal{F}},\varphi=u\; \mu\text{-a.e. on }F,u\in \mathcal{F}_\mathrm{e}.
\end{aligned}
\]
and
\[
\begin{aligned}
	&\check{\mathcal{F}}^{(s)}=\{\varphi\in L^2(F,\mu): \varphi=u\; \mu\text{-a.e. on }F\text{ for some }u\in \mathcal{F}^{(s)}_\text{e}\},	\\
	&\check{\mathcal{E}}^{(s)}(\varphi,\varphi)=\mathcal{E}^{(s)}(H^{(s)}_Fu,H^{(s)}_Fu),\quad \varphi\in \check{\mathcal{F}}^{(s)},\varphi=u\; \mu\text{-a.e. on }F,u\in \mathcal{F}^{(s)}_\text{e}.
\end{aligned}
\]
Note that since $(\mathcal{E}^{(s)},\mathcal{F}^{(s)})$ is a regular Dirichlet subspace of $(\mathcal{E,F})$, it follows from Lemma~2 of  \cite{LY} that $\mathcal{F}^{(s)}_\text{e}$ is a proper subset of $\mathcal{F}_\mathrm{e}$. Thus $\check{\mathcal{F}}^{(s)}\subset \check{\mathcal{F}}$.   We shall prove later that $(\check{\mathcal{E}}^{(s)},\check{\mathcal{F}}^{(s)})$ is actually a proper regular Dirichlet subspace of $(\check{\mathcal{E}},\check{\mathcal{F}})$ on $L^2(F,\mu)$.

Note that the global property (recurrent or transient) of $(\check{\mathcal{E}}^{(s)},\check{\mathcal{F}}^{(s)})$ (resp. $(\check{\mathcal{E}},\check{\mathcal{F}})$) is the same as that of $(\mathcal{E}^{(s)},\mathcal{F}^{(s)})$ (resp. $(\mathcal{E,F})$). In particular, $(\check{\mathcal{E}},\check{\mathcal{F}})$ is recurrent. On the other hand $(\check{\mathcal{E}},\check{\mathcal{F}})$ is irreducible. In fact, for any $\varphi=u|_F\in \check{\mathcal{F}}_\text{e}$ such that $\check{\mathcal{E}}(\varphi,\varphi)=0$ for some $u\in \mathcal{F}_\mathrm{e}$, it follows that $$\mathcal{E}(H_Fu,H_Fu)=0.$$ Thus $H_Fu\equiv C$ for some constant $C$ and $\varphi=u|_F=(H_Fu)|_F\equiv C$. From Theorem~5.2.16 of \cite{CM} we obtain that $(\check{\mathcal{E}},\check{\mathcal{F}})$ is irreducible.
 For $(\check{\mathcal{E}}^{(s)},\check{\mathcal{F}}^{(s)})$ we can also deduce similarly that every $\varphi\in \check{\mathcal{F}}^{(s)}_\text{e}$ with $\check{\mathcal{E}}^{(s)}(\varphi,\varphi)=0$ is also a constant function. Hence if $(\mathcal{E}^{(s)},\mathcal{F}^{(s)})$ is recurrent then $(\check{\mathcal{E}}^{(s)},\check{\mathcal{F}}^{(s)})$ is irreducible and recurrent. Moreover, the $\mu$-polar set with respect to $\check{X}$ or $\check{X}^{(s)}$ is only the empty set (see Theorem~5.2.8 of \cite{CM}).

Let us present the main result of this section, which tells that a Dirichlet form with non-trivial local part may have a regular Dirichlet subspace having no local part. Recall that $G$ can be written as \eqref{EQGCN} by assumption \textbf{(H)}.

\begin{theorem}\label{THM5}
Let $(\check{\mathcal{E}},\check{\mathcal{F}})$ and  $(\check{\mathcal{E}}^{(s)},\check{\mathcal{F}}^{(s)})$ be the traces of  $(\frac{1}{2}\mathbf{D},H^1(\mathbb{R}))$ and $(\mathcal{E}^{(s)},\mathcal{F}^{(s)})$ on $F$ relative to $\mu$ respectively. 
\begin{itemize}
\item[(1)] The Dirichlet form $(\check{\mathcal{E}}^{(s)},\check{\mathcal{F}}^{(s)})$ is a proper regular Dirichlet subspace of $(\check{\mathcal{E}},\check{\mathcal{F}})$ on $L^2(F,\mu)$.
\item[(2)] The Dirichlet form $(\check{\mathcal{E}},\check{\mathcal{F}})$ is a mixed-type Dirichlet form with the jumping part and for any $\varphi\in \check{\mathcal{F}}_\text{e}$,
\[
	\check{\mathcal{E}}(\varphi,\varphi)=\frac{1}{2}\int_F \varphi'(x)^2dx+\frac{1}{2}\sum_{n\geq 1}\frac{\left(\varphi(a_n)-\varphi(b_n)\right)^2}{|a_n-b_n|}.
	\]
	Its regular Dirichlet subspace $(\check{\mathcal{E}}^{(s)},\check{\mathcal{F}}^{(s)})$ is a non-local Dirichlet form whose jumping measure is the same as above and for any $\varphi\in \check{\mathcal{F}}^{(s)}_\text{e}$,
\[
	\check{\mathcal{E}}^{(s)}(\varphi,\varphi)=\check{\mathcal{E}}(\varphi,\varphi)=\frac{1}{2}\sum_{n\geq 1}\frac{\left(\varphi(a_n)-\varphi(b_n)\right)^2}{|a_n-b_n|}.
\]
\end{itemize}
\end{theorem}
\begin{proof}
\begin{itemize}
\item[(1)] 	Since $\mathcal{F}^{(s)}_\mathrm{e}$ is a subset of $\mathcal{F}_\mathrm{e}$, it follows that $\check{\mathcal{F}}^{(s)}$ is also a subset of $\check{\mathcal{F}}$. Note that it is a proper subset because
	\[
		f(x):=x,\quad x\in F
	\]
is locally in $\check{\mathcal{F}}$ but not locally in $\check{\mathcal{F}}^{(s)}$. Thus it suffices to prove that for any $u\in \mathcal{F}^{(s)}_\text{e}\subset\mathcal{F}_\text{e}$, it holds that $H_Fu=H^{(s)}_Fu$. In fact, $H^{(s)}_Fu$ is the unique function in $\mathcal{F}^{(s)}_\text{e}$ such that
\[
	\mathcal{E}(u,w)=0
\]
for any $w\in \mathcal{F}^{(s)}_{\text{e},G}$. It follows from Lemma~\ref{COR3} that $H^{(s)}_Fu$ is in $\mathcal{F}_\text{e}$ and
\[
	\mathcal{E}(u,w)=0
\]
for any $w\in \mathcal{F}_{\text{e},G}$. Thus $H^{(s)}_Fu=H_Fu$. 
\item[(2)] 
At first we assert that they both have no killing inside. In fact since $(\check{\mathcal{E}},\check{\mathcal{F}})$ is recurrent it is also conservative. Thus its life time $\check{\zeta}$ is always infinite. In particular $(\check{\mathcal{E}},\check{\mathcal{F}})$ has no killing inside. It follows that its regular Dirichlet subspace $(\check{\mathcal{E}}^{(s)},\check{\mathcal{F}}^{(s)})$ also has no killing inside.

 We refer the Feller measures of trace Dirichlet forms to \S5.5 of \cite{CM} and \cite{CFJ}. From Theorem~1 of  \cite{LY} we can deduce that $(\mathcal{E,F})$ and $(\mathcal{E}^{(s)},\mathcal{F}^{(s)})$  have the same Feller measures for $F$ because they are exactly the jumping measures of $(\check{\mathcal{E}},\check{\mathcal{F}})$ and $(\check{\mathcal{E}}^{(s)},\check{\mathcal{F}}^{(s)})$. Denote the common Feller measure on $F\times F$ by $U(dx dy)$. Then for any $\varphi\in \check{\mathcal{F}}_\text{e}$ (see (5.6.7) of \cite{CM}),
\[
	\check{\mathcal{E}}(\varphi,\varphi)=\frac{1}{2}\mu_{\langle H_F\varphi\rangle}(F)+\frac{1}{2}\int_{F\times F}(\varphi(x)-\varphi(y)))^2U(dxdy),
\]
where $\mu_{\langle H_F\varphi\rangle}$ is the energy measure of $(\mathcal{E,F})$ relative to $H_F\varphi$ and for any $\phi\in \check{\mathcal{F}}^{(s)}_\text{e}$,
\[
	\check{\mathcal{E}}^{(s)}(\phi,\phi)=\frac{1}{2}\mu^{(s)}_{\langle H^{(s)}_F\phi\rangle}(F)+\frac{1}{2}\int_{F\times F}(\phi(x)-\phi(y)))^2U(dxdy),
\]
where  $\mu^{(s)}_{\langle H^{(s)}_F\phi\rangle}$ is the energy measure of $(\mathcal{E}^{(s)},\mathcal{F}^{(s)})$ relative to $H^{(s)}_F\phi$. Note that the first terms in the right sides of  above two equations are the strongly local part of  corresponding Dirichlet forms.

We claim that for any $u\in \mathcal{F}_\text{e}$, the energy measure
	\begin{equation}\label{EQMLUR}
	\mu_{\langle u\rangle}=u'(x)^2dx
	\end{equation}
	and for any $v\in \mathcal{F}^{(s)}_\text{e}$, the energy measure
	\[
		\mu^{(s)}_{\langle v\rangle}=(\frac{dv}{ds})^2ds.
	\]
	In particular $\mu^{(s)}_{\langle v\rangle}(F)=0$ for any $v\in \mathcal{F}^{(s)}_\text{e}$.

In fact for any $f\in C_c^1(\mathbb{R})$ we have (see \S3.2 of  \cite{FU})
	\[
		\int_\mathbb{R}fd\mu_{\langle u\rangle}=2\mathcal{E}(uf,u)-\mathcal{E}(u^2,f)=\int_\mathbb{R}f(x)u'(x)^2dx.
	\]
	Thus $\mu_{\langle u\rangle}=u'(x)^2dx$. Similarly we can prove that $\mu^{(s)}_{\langle v\rangle}=(dv/ds)^2ds$. In particular it follows from \eqref{EQSFX} that for any $v\in \mathcal{F}^{(s)}_\text{e}$,
	\[
		\mu^{(s)}_{\langle v\rangle}(F)=\int_F (\frac{dv}{ds})^2ds=\int_F(\frac{dv}{ds})^21_G(x)dx=0.
	\]

Moreover fix $\varphi\in \check{\mathcal{F}}_\text{e}$. Since $H_F\varphi=\varphi$ on $F$, similar to the proof of Lemma~\ref{THM1} we have
\[
(H_F\varphi)'=\varphi',\quad \text{a.e. on }F.
\]
Then it follows from \eqref{EQMLUR} that 
\[
	\mu_{\langle H_F\varphi\rangle}(F)=\int_F \left((H_F\varphi)'\right)^2dx=\int_F \varphi'(x)^2dx.
	\]

Finally, we shall compute the Feller measure $U$. Recall that in \S\ref{INTRO} we set $G=\bigcup_{n\geq 1}I_n$, where $\{I_n=(a_n,b_n):n\geq 1\}$ is a series of disjoint open intervals without common endpoints. 
Fix two non-negative and bounded functions $\varphi$ and $\phi$ on $F$ such that $\varphi\cdot \phi\equiv 0$. We set $\varphi(-\infty)=\varphi(\infty)=\phi(-\infty)=\phi(\infty)=0$ for convenience. It follows from (5.5.13) and (5.5.14) of \cite{CM} that
\[
	U(\varphi \otimes \phi)=\uparrow \lim_{\alpha\uparrow \infty}\alpha(H^\alpha_F\varphi,H_F\phi)_{G}=\uparrow \lim_{\alpha\uparrow \infty}\sum_{n\geq 1}\alpha(H^\alpha_F\varphi,H_F\phi)_{I_n},
\]
where $H^\alpha_F\varphi(x):=E^x\left( \text{e}^{-\alpha \sigma_F}\varphi(X_{\sigma_F})\right)$ for any $x\in \mathbb{R}$. Fix a finite component $I_n$ of $G$ and $x\in I_n$. Since the trajectories of Brownian motion are continuous, it follows that
\[
	X_{\sigma_F}=a_n\text{ or }b_n,\quad P^x\text{-a.s.}
\]
Hence
\[
\begin{aligned}
	H_F\phi(x)&=\phi(a_n)\cdot P^x(X_{\sigma_F}=a_n)+ \phi(b_n)\cdot P^x(X_{\sigma_F}=b_n)\\
	 &=\phi(a_n)\cdot\frac{b_n-x}{b_n-a_n}+\phi(b_n)\cdot\frac{x-a_n}{b_n-a_n}
	\end{aligned}\]
and $$
\begin{aligned}
	H^\alpha_F\varphi(x)&=\varphi(a_n)\cdot E^x(\text{e}^{-\alpha \sigma_F},X_{\sigma_F}=a_n) \\
	  &\qquad \quad + \varphi(b_n)\cdot E^x(\text{e}^{-\alpha \sigma_F},X_{\sigma_F}=b_n).
	\end{aligned}$$
Otherwise if $I_n$ is infinite, i.e. $a_n=-\infty$ or $b_n=\infty$,  then $X_{\sigma_F}$ is located at the finite endpoint of $I_n$ $P^x$-a.s. for any $x\in I_n$. However $\varphi(a_n)\phi(a_n)=\varphi(b_n)\phi(b_n)=0$. It follows that $(H^\alpha_F\varphi,H_F\phi)_{I_n}=0$.
Set
\[
\begin{aligned}
r_n(x)&:=\frac{b_n-x}{b_n-a_n},\\
p_n(x)&:=E^x(e^{-\alpha \sigma_F};X_{\sigma_F}=a_n),\\
q_n(x)&:=E^x(e^{-\alpha \sigma_F};X_{\sigma_F}=b_n).\end{aligned}\]
Then we have
\[
\begin{aligned}
	U(\varphi \otimes \phi)=\uparrow \lim_{\alpha\uparrow \infty}&\sum_{n\geq 1} \alpha\bigg(\varphi(a_n)\phi(b_n)\int_{I_n}p_n(x)(1-r_n(x))dx \\
	 &+\varphi(b_n)\phi(a_n)\int_{I_n} q_n(x)r_n(x)dx\bigg).
\end{aligned}\]
and $U$ is supported on a set of $\mathbb{R}^2$ containing countable points $$\{(a_n,b_n),(b_n,a_n):a_n>-\infty, b_n<\infty, n\geq 1\}.$$
Let $\varphi=1_{a_n},\phi=1_{b_n}$, where $a_n>-\infty$ and $b_n<\infty$. Note that $$p_n(x)=\frac{\sinh \sqrt{2\alpha} (b_n-x)}{\sinh \sqrt{2\alpha} (b_n-a_n)}$$ and $$q_n(x)=\frac{\sinh \sqrt{2\alpha} (x-a_n)}{\sinh \sqrt{2\alpha} (b_n-a_n)},$$ see Problem~6 in \S1.7 of \cite{IM}. Then we obtain that
\[
	U((a_n,b_n))=\lim_{\alpha\uparrow \infty}\alpha\int_{a_n}^{b_n} p_n(x)(1-r_n(x))dx=\frac{1}{2d_n},
\]
where $d_n=|b_n-a_n|$. Clearly $$U((b_n,a_n))=U((a_n,b_n))=\frac{1}{2d_n}.$$ When $a_n=-\infty$ or $b_n=\infty$ we still denote
$
	U((a_n,b_n)):=\frac{1}{2d_n}=0.
$
That completes the proof.
\end{itemize}
\end{proof}

\begin{remark}
\begin{itemize}
\item[(1)] The open set $G$ in \eqref{EQSFX} is an essential  characteristic of regular Dirichlet subspace $(\mathcal{E}^{(s)},\mathcal{F}^{(s)})$ of $(\mathcal{E,F})$. As stated in Lemma~\ref{THM1}, before leaving $G$, $X^{(s)}$ is equivalent to $X$. The above theorem shows that the difference between $X^{(s)}$ and $X$ is located on their traces on the boundary $F$ of $G$. In fact the trace Dirichlet form $(\check{\mathcal{E}}^{(s)},\check{\mathcal{F}}^{(s)})$ is still a proper regular Dirichlet subspace of $(\check{\mathcal{E}},\check{\mathcal{F}})$.
\item[(2)] 	Denote the extended Dirichlet spaces of $(\check{\mathcal{E}}^{(s)},\check{\mathcal{F}}^{(s)})$ and $(\check{\mathcal{E}},\check{\mathcal{F}})$ by $\check{\mathcal{F}}^{(s)}_\text{e}$ and $\check{\mathcal{F}}_\text{e}$. Clearly
\begin{equation}\label{EQCMFS}
	\check{\mathcal{F}}^{(s)}_\text{e}=	\mathcal{F}^{(s)}_\text{e}\big|_F,\quad \check{\mathcal{F}}_\text{e}=\mathcal{F}_\text{e}\big|_F.
\end{equation}
Here for a class $\mathcal{C}$ of functions on $\mathbb{R}$,
\[
	\mathcal{C}|_F:=\{u|_F:u\in\mathcal{C}\},
\]
where $u|_F$ is the restriction of $u$ on $F$.
Note that the extended Dirichlet spaces are independent of the choice of $\mu$. More precisely, for any Radon measure $\mu$ on $\mathbb{R}$ with the support $F$, their extended Dirichlet spaces are always given by \eqref{EQCMFS}. Thus the results of Theorem~\ref{THM5}~(1) can essentially be expressed as
\[
	\check{\mathcal{F}}^{(s)}_\text{e}\subset \check{\mathcal{F}}_\text{e},\quad \check{\mathcal{E}}(u,v)=\check{\mathcal{E}}^{(s)}(u,v),\quad u,v\in \check{\mathcal{F}}^{(s)}_\text{e}.
\]
In particular $\check{\mathcal{F}}^{(s)}_\text{e}$ is a proper subset of $\check{\mathcal{F}}_\text{e}$.
\item[(3)] In Corollary~2 of  \cite{LY} we have proved that if $(\mathcal{E,F})$ is a L\'{e}vy type Dirichlet form whose strongly local part does not vanish, then neither does the strongly local part of any regular Dirichlet subspace of $(\mathcal{E,F})$. The above theorem implies that this fact is not always right.
\end{itemize}
\end{remark}

\section{Orthogonal complement and darning processes}\label{OCO}

As stated in Theorem~\ref{THM5} the regular Dirichlet subspace $(\check{\mathcal{E}}^{(s)},\check{\mathcal{F}}^{(s)})$ only contains the non-local information of $(\check{\mathcal{E}},\check{\mathcal{F}})$. An interesting question is whether (and how) the `orthogonal complement' of $(\check{\mathcal{E}}^{(s)},\check{\mathcal{F}}^{(s)})$ contains the remaining information, i.e. the strongly local part, of $(\check{\mathcal{E}},\check{\mathcal{F}})$.

\subsection{Orthogonal complement}

Note that $\check{\mathcal{F}}_\mathrm{e}$ (resp. $\check{\mathcal{F}}^{(s)}_\mathrm{e}$) is the restriction of $\mathcal{F}_\mathrm{e}$ (resp. $\mathcal{F}^{(s)}_\mathrm{e}$) on $F$, say \eqref{EQCMFS}. In order to determine the orthogonal complement of trace subspace,
we shall first consider the orthogonal complement of $\mathcal{F}^{(s)}_\text{e}$ in $\mathcal{F}_\text{e}$ relative to the quadratic form $\mathcal{E}(\cdot, \cdot)$.
 Although $\mathcal{F}_\text{e}$ is not a Hilbert space relative to the quadratic form $\mathcal{E}(\cdot, \cdot)$, we can still define the orthogonal complement of $\mathcal{F}^{(s)}_\text{e}$ in $\mathcal{F}_\text{e}$ formally by
	\begin{equation}\label{EQGSU}
	\mathcal{G}^{(s)}:=\{u\in \mathcal{F}_\text{e}:\mathcal{E}(u,v)=0\text{ for any }v\in \mathcal{F}^{(s)}_\text{e}\}.
\end{equation}
Before characterizing $\mathcal{G}^{(s)}$ we need to make some discussions on $\mathcal{F}^{(s)}_\text{e}$.
{Both $\infty$ and $-\infty$ are non-regular boundaries of $\mathbb{R}$ for the Dirichlet form $(\mathcal{E}^{(s)},\mathcal{F}^{(s)})$ on $L^2(\mathbb{R})$. Hence, by virtue of Theorem~2.2.11~(ii) of \cite{CM}, }
\[
\begin{aligned}
	\mathcal{F}^{(s)}_\text{e}=\{u:u&\ll s, \int_{\mathbb{R}}\left(\frac{du}{ds}\right)^2ds<\infty,\\ &u(\infty)=0\text{ if }s(\infty)<\infty\text{ and }u(-\infty)=0\text{ if }s(-\infty)>-\infty\}.
\end{aligned}\]
{We can make the following classification of the boundaries $\infty, -\infty$: }
\begin{description}
\item[Case I.] $s(-\infty)=-\infty, s(\infty)=\infty$,
\item[Case II.] $s(-\infty)>-\infty, s(\infty)=\infty$ or $s(-\infty)=-\infty, s(\infty)<\infty$,
\item[Case III.] $s(-\infty)>-\infty, s(\infty)<\infty$.
\end{description}
{Note that $(\mathcal{E}^{(s)}, \mathcal{F}^{(s)})$ is recurrent in Case I but transient in Case II and III. Note further that, for Case I and II, we have $m(G)=\infty$. }

We have the following lemma similar to Lemma~\ref{LM3} to characterize $\mathcal{F}^{(s)}_\text{e}$.

\begin{lemma}\label{LM5}
It holds that
\[
\begin{aligned}
	\mathcal{F}^{(s)}_\text{e}=\{u&\in \mathcal{F}_\text{e}:u'=0\text{ a.e. on }F,\\ &u(\infty)=0\text{ if }s(\infty)<\infty\text{ and }u(-\infty)=0\text{ if }s(-\infty)>-\infty\}.
	\end{aligned}
\]
\end{lemma}

Now we shall give a useful expression of $\mathcal{G}^{(s)}$ and an `orthogonal' decomposition of $\mathcal{F}_\mathrm{e}$ relative to $\mathcal{F}^{(s)}_\mathrm{e}$ and $\mathcal{G}^{(s)}$. The decomposition \eqref{EQUUU} indicates that the definition of $\mathcal{G}^{(s)}$ in \eqref{EQGSU} is reasonable.

\begin{theorem}\label{THM3}
	The class $\mathcal{G}^{(s)}$ has the following characterization
	\begin{equation}\label{EQGSH}
		\mathcal{G}^{(s)}=\big\{u\in \mathcal{F}_\text{e}:u'\text{ is a constant a.e. on }G\big\}.
	\end{equation}
In particular for Case I and II, equivalently $m(G)=\infty$, it holds that
	\begin{equation}\label{EQMGSU}
		\mathcal{G}^{(s)}=\{u\in \mathcal{F}_\text{e}:u'=0\text{ a.e. on }G\}.
	\end{equation}
Moreover, any  $u\in \mathcal{F}_\text{e}$ can be expressed  as
	\begin{equation}\label{EQUUU}
		u=u_1+u_2,\quad u_1\in \mathcal{F}^{(s)}_\text{e},u_2\in \mathcal{G}^{(s)}.
	\end{equation}
This decomposition is unique if $(\mathcal{E}^{(s)},\mathcal{F}^{(s)})$ is transient and unique up to a constant if $(\mathcal{E}^{(s)},\mathcal{F}^{(s)})$ is recurrent.
\end{theorem}
\begin{proof}
First we shall prove the characterization \eqref{EQGSH} of $\mathcal{G}^{(s)}$.
Take a function $u\in \mathcal{F}_\text{e}$ such that $u'=C$ a.e. on $G$, where $C$ is a constant. For Case I and II, since $m(G)=\infty$ it follows that $C=0$. Then for any $v\in \mathcal{F}^{(s)}_\text{e}$ we have
\[
	\mathcal{E}(u,v)=\frac{1}{2}\int_\mathbb{R} u'(x)v'(x)dx.
\]
From Lemma~\ref{LM5} we know that $v'=0$ a.e. on $F$ whereas $u'=0$ a.e. on $G$. Hence we obtain that $\mathcal{E}(u,v)=0$. For Case III, since $(\mathcal{E}^{(s)},\mathcal{F}^{(s)})$ is transient we know that $C_c^1\circ s:=\{\varphi\circ s:\varphi\in C_c^1(s(\mathbb{R}))\}$ is $\mathcal{E}$-dense in $\mathcal{F}^{(s)}_\text{e}$. For any function $v\in C_c^1\circ s$, clearly $v'=0$ a.e. on $F$. Thus we have
\[
	\mathcal{E}(u,v)=\frac{1}{2}\int_G u'(x)v'(x)dx=\frac{C}{2}\int_\mathbb{R}v'(x)dx=0.
\]
Generally for $v\in \mathcal{F}^{(s)}_\text{e}$ we can take a sequence $\{v_n:n\geq 1\}$ in $C_c^1\circ s$ such that $\mathcal{E}(v_n-v,v_n-v)\rightarrow 0$ as $n\rightarrow \infty$. Since $\mathcal{E}(u,v_n)=0$ we can deduce that $\mathcal{E}(u,v)=0$.

{To the contrary}, take a function $u\in \mathcal{G}^{(s)}$. Since $C_c^1\circ s\subset \mathcal{F}^{(s)}_\text{e}$ it follows that
\[
	\mathcal{E}(u,v)=0,\quad v\in C_c^1\circ s.
\]
Let $t$ be the inverse function of $s$, i.e. $t=s^{-1}$. Then
\[
	\int_{s(\mathbb{R})}u'(t(x))\varphi'(x)dx=0
\]
for any $\varphi\in C_c^1(s(\mathbb{R}))$. It follows that $u'\circ t$ is a constant a.e. on $s(\mathbb{R})$. Denote all of such $x\in s(\mathbb{R})$ by $J$, i.e. $u'\circ t$ is a constant on $J$.  Let $\tilde{J}:=t(J)$. Then $u'$ is a constant on $\tilde{J}$. On the other hand
\[
	m(G\setminus \tilde{J})=\int_{\tilde{J}^c}1_G(x)dx=\int_{\tilde{J}^{c}}ds(x)=m(s(\mathbb{R})\setminus J)=0.
\]
Thus $u'$ is a constant a.e. on $G$. Thus \eqref{EQGSH} is proved.

{Note that any function $u\in \mathcal{F}_\mathrm{e}$ satisfies $u'\in L^2(\mathbb{R})$.} In particular, if $m(G)=\infty$,
then it follows that any function $u$ in $\mathcal{G}^{(s)}$ satisfies that $u'=0$ a.e. on $G$, i.e. \eqref{EQMGSU} is proved.

Finally we shall construct the decomposition \eqref{EQUUU} for any $u\in \mathcal{F}_\mathrm{e}$. Assume $C_0=u(0)$. For any $x\in \mathbb{R}$,
	\[
	\begin{aligned}
		u(x)-C_0=\int_0^x u'(y)dy=\int_0^x u'(y)1_G(y)dy+\int_0^x u'(y)1_F(y)dy.
	\end{aligned}\]
First for Case I, define
\begin{equation}\label{EQUXI}
	u_1(x)=\int_0^x u'(y)1_G(y)dy,\quad x\in \mathbb{R}
\end{equation}
and $u_2=u-u_1$. It follows from {Lemma~\ref{LM5} and  the first assertion of Theorem~\ref{THM3}} that $u_1\in \mathcal{F}^{(s)}_\text{e}$ and $u_2\in \mathcal{G}^{(s)}$. Secondly for Case II without loss of generality assume that {$s(-\infty)>-\infty$} but $s(\infty)=\infty$. Then
\[
	\bigg|\int_{-\infty}^0 u'(y)1_G(y)dy\bigg |^2 \leq \int_{-\infty}^0 u'(y)^2dy \cdot (s(0)-s(-\infty))<\infty.
\]
Let $M_{-\infty}:=\int_{-\infty}^0 u'(y)1_G(y)dy$ which is a finite constant and define
\[
	u_1(x):=\int_0^x u'(y)1_G(y)dy+M_{-\infty},\quad x\in \mathbb{R}
\]
and $u_2:=u-u_1$. It follows that
\[
	\lim_{x\rightarrow -\infty}u_1(x)=-\int_{-\infty}^0u'(y)1_G(y)dy+M_{-\infty}=0.
\]
Thus we can also deduce that $u_1\in \mathcal{F}^{(s)}_\text{e}$ and $u_2\in \mathcal{G}^{(s)}$. Finally for Case III  we can similarly deduce that
\[
	M:=\int_{-\infty}^\infty u'(y)1_G(y)dy
\]
is finite. Let $C_1:=M/(s(\infty)-s(-\infty))$ and
\[
	C_2:=\int_{-\infty}^0(u'(y)-C_1)1_G(y)dy
\]
which are both finite constants. Define
\[
	u_1(x):=\int_0^x (u'(y)-C_1)1_G(y)dy+C_2,\quad x\in \mathbb{R}
\]
and $u_2:=u-u_1$. Note that
\[
	\lim_{x\rightarrow -\infty}u_1(x)=-\int_{-\infty}^0(u'(y)-C_1)1_G(y)dy+C_2=0
	\]
	and
\[
	\lim_{x\rightarrow \infty}u_1(x)=\int_{-\infty}^\infty(u'(y)-C_1)1_G(y)dy=0.
\]
Hence it follows that $u_1\in \mathcal{F}^{(s)}_\text{e}$. We claim that $u_2\in \mathcal{G}^{(s)}$. In fact for a.e. $x\in G$,
\[
	u'_2(x)=u'(x)-u'_1(x)=u'(x)-(u'(x)-C_1)1_G(x)=C_1.
\]
It follows from Theorem~\ref{THM3} that $u_2\in \mathcal{G}^{(s)}$.

Now assume $u\in \mathcal{F}^{(s)}_\text{e}\cap \mathcal{G}^{(s)}$. It follows from \eqref{EQGSU} that $\mathcal{E}(u,u)=0$. Thus $u\equiv C$ for some constant $C$. If $(\mathcal{E}^{(s)},\mathcal{F}^{(s)})$ is transient, since $u\in \mathcal{F}^{(s)}_\text{e}$ we have $\lim_{x\rightarrow -\infty \text{ or }\infty}u(x)=0$. Thus $C=0$ and $u\equiv 0$. Otherwise $(\mathcal{E}^{(s)},\mathcal{F}^{(s)})$ is recurrent, $C$ is not necessarily $0$. In fact when defining $u_1$ for this case,  the decomposition is still valid if we add any constant to right side of \eqref{EQUXI}. Therefore the decomposition \eqref{EQUUU} is unique up to a constant when $(\mathcal{E}^{(s)},\mathcal{F}^{(s)})$ is recurrent.
\end{proof}

\begin{corollary}
	Let $u\in \mathcal{F}_\text{e}$ and $u_2$ the function in the decomposition \eqref{EQUUU}. Then $u'=u'_2$ a.e. on $F$.
\end{corollary}

We present the following decomposition similar to \eqref{EQUUU} for the functions in $\mathcal{H}_F$.

\begin{proposition}\label{PROP3}
	Any $u\in \mathcal{H}_F$ can be expressed as
	\[
		u=u_1+u_2
	\]
for some $u_1\in \mathcal{H}^{(s)}_F$ and $u_2\in \mathcal{G}^{(s)}$. This decomposition is unique if $(\mathcal{E}^{(s)},\mathcal{F}^{(s)})$ is transient and  unique up to a constant if $(\mathcal{E}^{(s)},\mathcal{F}^{(s)})$ is recurrent. In particular,
\begin{equation}\label{EQGSUH}
	\mathcal{G}^{(s)}=\{u\in \mathcal{H}_F: \mathcal{E}(u,v)=0\text{ for any } v\in \mathcal{H}^{(s)}_F\}.
\end{equation}
\end{proposition}
\begin{proof}
We first prove \eqref{EQGSUH}. Since $\mathcal{F}_{\text{e},G}=\mathcal{F}^{(s)}_{\text{e},G}\subset \mathcal{F}^{(s)}_\text{e}$ and $\mathcal{H}^{(s)}_F\subset \mathcal{F}^{(s)}_\text{e}$, it follows that for any $u\in \mathcal{G}^{(s)}$, $u\in \mathcal{H}_F$ and $\mathcal{E}(u,v)=0$ for any  $v\in \mathcal{H}^{(s)}_F$. {To the contrary}, let $u$ be a function in the right side of \eqref{EQGSUH} and $w$ a function in $\mathcal{F}^{(s)}_\text{e}$. Suppose that $w=w_1+w_2$ is the decomposition of $w$ in \eqref{EQVVV}. Since $u\in \mathcal{H}_F$ and $w_1\in \mathcal{F}^{(s)}_{\text{e},G}=\mathcal{F}_{\text{e},G}$ we have $\mathcal{E}(u,w_1)=0$. Moreover $\mathcal{E}(u,w_2)=0$ is also clear.

Now for any $u\in \mathcal{H}_F\subset \mathcal{F}_\text{e}$, it can be expressed as
\[
	u=u_1+u_2
\]
for some $u_1\in \mathcal{F}^{(s)}_\text{e}$ and $u_2\in \mathcal{G}^{(s)}$. We claim that $u_1\in \mathcal{H}^{(s)}_F$. To this end, since $u_2\in \mathcal{G}^{(s)}\subset \mathcal{H}_F$ it follows that $u_1\in \mathcal{H}_F$. Thus for any $w\in \mathcal{F}^{(s)}_{\text{e},G}=\mathcal{F}_{\text{e},G}$ we can deduce that
\[
	\mathcal{E}^{(s)}(u_1,w)=\mathcal{E}(u_1,w)=0.
\]
Hence $u_1\in \mathcal{H}^{(s)}_F$. The uniqueness can be proved through the similar way to Theorem~\ref{THM3}.  
\end{proof}

\begin{remark}
	Note that if $(\mathcal{E}^{(s)},\mathcal{F}^{(s)})$ is transient then any $u\in \mathcal{F}_\text{e}$ can be expressed uniquely as
	\[
		u=u_1+u_2+u_3,
	\]
where $u_1\in \mathcal{F}^{(s)}_{\text{e},G}=\mathcal{F}_{\text{e},G}, u_2\in \mathcal{H}^{(s)}_F, u_3\in \mathcal{G}^{(s)}$ and $\hat{u}_2:=u_2+u_3\in \mathcal{H}_F$. This decomposition is similar to the orthogonal decomposition with respect to the quadratic form $\mathcal{E}$ whereas $\mathcal{F}_\text{e}$ is not a Hilbert space. Otherwise when $(\mathcal{E}^{(s)},\mathcal{F}^{(s)})$ is  recurrent we also have such kind of decomposition. In particular $u_1$ and $\hat{u}_2$ are unique but $u_2$ and $u_3$ are only unique up to a constant.
\end{remark}

\subsection{Regular representation via darning transform}\label{SEC32}

In the rest of this section, we always assume that $m(G)=\infty$, i.e.  
\[
	s(-\infty)=-\infty \text{ or } s(\infty)=\infty.
\]
It follows from Theorem~\ref{THM3} that
\[
	\mathcal{G}^{(s)}=\{u\in \mathcal{F}_\text{e}:u'=0\text{ a.e. on }G\}.
\]
That means for any function $u\in \mathcal{G}^{(s)}$ and $I_n$ a component of $G$, $u$ is a constant on $I_n$. Let $\mathcal{G}^{(s)}_0:=\mathcal{G}^{(s)}\cap L^2(\mathbb{R})=\{u\in \mathcal{F}:u'=0\text{ a.e. on }G\}$. In fact $\mathcal{G}^{(s)}_0$ is very close to a Dirichlet space.

Let $r_-=\inf\{x: x\in F\}$, $r_+:=\sup\{x:x\in F\}$ and 
\begin{equation}\label{EQIRR}
I:=(r_-,r_+).
\end{equation}
 If $r_->-\infty$ (resp. $r_+<\infty$), then $u=0$ on $(-\infty, r_-]$ (resp. $[r_+,\infty)$) for any $u\in \mathcal{G}^{(s)}_0$. 
Hence the quadratic form $(\mathcal{E}, \mathcal{G}^{(s)}_0)$ on $L^2(\mathbb{R})$ can be identified with { the one} on $L^2(I,m_I)$, where $m_I$ is the restriction to $I$ of the Lebesgue measure and {$\mathcal{G}^{(s)}_0$} can be written as
{\[
	\mathcal{G}^{(s)}_0=\{u\in H^1_{0,\mathrm{e}}(I)\cap L^2(I,m_I): u'=0\text{ a.e. on } G\cap I\},
\]
where 
\[
	H^1_{0,\mathrm{e}}=\{u\in H^1_\mathrm{e}(I): u(r_{\pm})=0 \text{ whenever } r_{\pm} \text{ is finite} \}. 
\]
Then we have the following lemma.}

{
\begin{lemma}\label{LM6}
The form $(\mathcal{E},\mathcal{G}^{(s)}_0)$ is a Dirichlet form on $L^2(I, m_I)$ in the wide sense, i.e. it satisfies all conditions of Dirichlet form except for the denseness of $\mathcal{G}^{(s)}_0$ in $L^2(I, m_I)$.
\end{lemma}

The closeness and Markovian property of $(\mathcal{E}, \mathcal{G}^{(s)}_0)$ on $L^2(I,m_I)$ can be deduced directly from the definitions. So the proof of Lemma~\ref{LM6} is trivial.}

We shall now find the regular representation of $(\mathcal{E}, \mathcal{G}^{(s)}_0)$. The
notion of regular representation of Dirichlet space was introduced by M. Fukushima in his cornerstone paper \cite{FU2}.
In his terminologies, a Dirichlet form in wide sense is also called a D-space and he denoted it by $(E,\xi, \mathcal{F},\mathcal{E})$,
where $(E,\xi)$ is the state space with a measure and $(\mathcal{E,F})$ is the Dirichlet form on $L^2(E,\xi)$ in wide sense. Due to Lemma~\ref{LM6},
{$(I,m_I, \mathcal{G}^{(s)}_0, \mathcal{E})$} is a D-space. Two D-spaces $(E,\xi, \mathcal{F},\mathcal{E})$ and $(E',\xi', \mathcal{F}',\mathcal{E}')$
 are called \emph{equivalent} if there is an algebraic isomorphism $\Phi$ from $\mathcal{F}\cap L^\infty(E)$ onto $\mathcal{F}'\cap L^\infty(E')$ and $\Phi$ preserves three kinds of metrics:
 \[
 	||u||_{\infty}=||\Phi u||_\infty,\ \mathcal{E}(u,u)=\mathcal{E}'(\Phi u,\Phi u) \text{ and } (u,u)_\xi=(\Phi u,\Phi u)_{\xi'}.
 \]
For a given D-space $(E,\xi, \mathcal{F},\mathcal{E})$, another  D-space $(E',\xi', \mathcal{F}',\mathcal{E}')$ is called a regular representation of $(E,\xi, \mathcal{F},\mathcal{E})$ if they are equivalent and $(\mathcal{E}',\mathcal{F}')$ is a regular Dirichlet form on $L^2(E',\xi')$. By using Gelfand representations of subalgebras of $L^\infty$, Fukushima  proved that regular representations always
 exist for any D-space in Theorem~2 of \cite{FU2}, but to find it is another story.

We shall introduce the `darning' transform on D-space {$(I,m_I,\mathcal{G}^{(s)}_0, \mathcal{E})$} to find its regular representation. This transform is in fact darning each {finite} component $I_n$ of $G$ and its endpoints into a whole part and regarding this whole part as a new `point' in the fresh state space. To be more precise, fix a point $z\in F\setminus H$, where $H$ is given by \eqref{EQHAB} and a surjective mapping $j$ from {$I$} to {$\mathbb{R}^*_0:=j(I)$} is defined through the following way: for any {$x\in I$},
\begin{equation}\label{EQJXI}
	j(x):=\int_z^x 1_F(t)dt.
\end{equation}
Note that $j(z)=0$ and $j$ is non-decreasing. Recall that $I=(r_-,r_+)$ is given by \eqref{EQIRR}.

\begin{lemma}\label{LM33}
Let $r^*_+:=j(r_+)$ and $r^*_-:=j(r_-)$. Then {$\mathbb{R}^*_0=(r^*_-,r^*_+)$}. Furthermore, $j(x)=j(y)$ if and only if $x,y\in \bar{I}_n$ for some {finite} component $I_n$ of $G$, where $\bar{I}_n$ is the closure of $I_n$.
\end{lemma}
\begin{proof}
The first assertion is obvious. For the second assertion, `if part' is also clear. For its `only if' part, let $x,y\in F$ such that $x<y$ and $(x,y)$ is not contained in a component of $G$, we claim that $j(x)<j(y)$. In fact, it follows that
 $F\cap (x,y)$ is not empty. Take $w\in F\cap (x,y)$ and let $d=|x-w|\wedge|w-y|$. Recall that for $0<\epsilon<d/2$, we have $m(F\cap (w-\epsilon,w+\epsilon))>0$. Thus we can deduce that
\[
	j(y)-j(x)=\int_x^y 1_F(t)dt \geq \int_{w-\epsilon}^{w+\epsilon}1_F(t)dt>0.
\]
That completes the proof. 
\end{proof}

\begin{figure}
\centering
\includegraphics[scale=0.45]{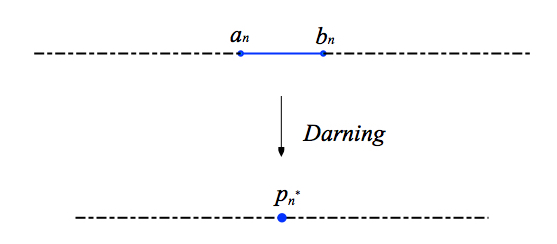}
\caption{Darning}\label{DAR}
\end{figure}

By Lemma~\ref{LM33}, we can denote 
\[
	p_n^*:=j(\bar{I}_n). 
\]
Then $\{p_n^*: n\geq 1\}$ is dense in {$\mathbb{R}^*_0$}, and $j|_{F\setminus H}$ is a bijective mapping from $F\setminus H$ to ${\mathbb{R}^*_0}\setminus \{p_n^*:n\geq 1\}$. See \textsc{Fig}~\ref{DAR} for an intuitive illustration. 

\begin{remark}
Note that $G$ may have no infinite components.
For example, let $K$ be a generalized Cantor set on $[0,1]$ such that $m(K)=\lambda>0$ and define $$F:=\bigcup_{k\in \mathbf{Z}}(K+2k),$$ where $K+2k:=\{x+2k:x\in K\}$. Then $G:=F^c$ is an open set whose components are all finite but $m(G)=\infty$. In this case, 
{we have $I=\mathbb{R}, \mathbb{R}^*_0=(r^*_-,r^*_+)$. }
\end{remark}

{Define a measure $m^*_0$ on $\mathbb{R}^*_0$ by
\begin{equation}\label{EQROJ}
\quad m^*_0=m_I\circ j^{-1},
\end{equation}
which is a Radon measure on $\mathbb{R}^*_0$.}
Heuristically,  $\mathbb{R}^*_0$ endowed with the measure $m^*_0$ is transformed from $\mathbb{R}$ with $m$ by darning the closure of each finite component of $G$ into a single point and getting rid of the (possible) infinite component of $G$. 

Since $j$ collapses each $\bar{I}_n$ into a `point', say $p^*_n$, and any function $u$ in $\mathcal{G}^{(s)}_0$  is exactly a constant on each $\bar{I}_n$, this function $u$ determines a unique function $\hat{u}$ on  {$\mathbb{R}^*_0$} through a `darning' method:
\begin{equation}\label{EQUJU}
 \hat{u}\circ j=u.
\end{equation}
Precisely, $\hat u$ is expressed as $\hat{u}(p_n^*):=u(a_n)$ or $u(b_n)$ for any $n\geq 1$ and $$\hat{u}|_{{\mathbb{R}^*_0}\setminus\{p^*_n:n\geq 1\}}:=u|_{F\setminus H}\circ j^{-1}.$$
Further define
\begin{equation}\label{EQGSUU}
\begin{aligned}
	\mathcal{G}^{(s)*}_0:=\{\hat u:\ u\in\mathcal{G}^{(s)}_0\}.
\end{aligned}
\end{equation}
Clearly $u\mapsto \hat u$ is a linear bijection between $\mathcal{G}^{(s)}_0$ and $\mathcal{G}^{(s)*}_0$.
For any $\hat{u},\hat{v}\in \mathcal{G}^{(s)*}_0$, define a form $\mathcal{E}^*$ by
\begin{equation}\label{EQEJUVM}
\mathcal{E}^*(\hat{u},\hat{v}):=\mathcal{E}(u,v)=\frac{1}{2}\int_\mathbb{R}u'(x)v'(x)dx.
\end{equation}
The form $(\mathcal{E}^*, \mathcal{G}^{(s)*}_0)$ is called the \emph{darning transform} of $(\mathcal{E},\mathcal{G}^{(s)}_0)$ induced by $j$. 
The following theorem mentions that $(\mathcal{E}^*,\mathcal{G}^{(s)*}_{0})$ is truly a Dirichlet form on $L^2(\mathbb{R}^*_0,m^*_0)$.

\begin{theorem}\label{prop5}
The quadratic form $(\mathcal{E}^*,\mathcal{G}^{(s)*}_0)$ defined by \eqref{EQGSUU} and \eqref{EQEJUVM} can be expressed as
\begin{equation}\label{EQMGS}
\begin{aligned}
	\mathcal{G}^{(s)*}_0&= H^1_{0,\mathrm{e}}({\mathbb{R}^*_0})\bigcap L^2(\mathbb{R}^*_0,m^*_0),  \\
	\mathcal{E}^*(\hat{u},\hat{v})&=\frac{1}{2}\int_{{\mathbb{R}^*_0}}\hat{u}'(x)\hat{v}'(x)dx,\quad \forall \hat{u},\hat{v}\in \mathcal{G}^{(s)*}_0,
\end{aligned}
\end{equation}	
where 
\[
	H^1_{0,\mathrm{e}}({\mathbb{R}^*_0})=\left\{\hat{u}\in H^1_\mathrm{e}({\mathbb{R}^*_0}): \hat{u}(r^*_\pm)=0 \text{ whenever } r^*_\pm \text{ is finite}\right\}. 
\]
Furthermore,  a regular representation of D-space {$(I,m_I,\mathcal{G}^{(s)}_0,\mathcal{E})$} can be realized by the regular local Dirichlet form $(\mathcal{E}^*,\mathcal{G}^{(s)*}_0)$ on  $L^2(\mathbb{R}^*_0,m^*_0)$. {Its associated diffusion process is an absorbing Brownian motion on $\mathbb{R}^*_0$ being time changed by its positive continuous additive functional with the Revuz measure $m^*_0$.}
\end{theorem}

\begin{proof}
We first prove $\mathcal{G}^{(s)*}_0= H^1_{0,\mathrm{e}}({\mathbb{R}^*_0})\bigcap L^2(\mathbb{R}^*_0,m^*_0)$. 
Take a function $\hat{u}$ in the right side and let $u:=\hat{u}\circ j$. To prove $\hat{u}\in \mathcal{G}^{(s)*}_0$, it suffices to prove that $u\in \mathcal{G}^{(s)}_0$. Since $j(I_n)=p_n^*$, it follows that $u$ is a constant on any component $I_n$ of $G_n$. On the other hand take any $x\in {I}$, we have
\[
	u(x)-u(z)=\hat{u}(j(x))-\hat{u}(j(z))=\int_z^x \hat{u}'\circ j(t)1_F(t)dt,
\]
where $z$ is the starting point when defining $j$ in \eqref{EQJXI}. It follows from $\hat{u}'\in L^2({\mathbb{R}^*_0})$ that $\hat{u}'\circ j\cdot 1_F\in L^2({I})$. Thus $u$ is absolutely continuous and $u'=\hat{u}'\circ j\cdot 1_F$ (in sense of a.e.) is in $L^2({I})$. In other words, $u\in \mathcal{F}_\mathrm{e}$ and $u'=0$ a.e. on $G$. That indicates $u\in \mathcal{G}^{(s)}$. Note that $\lim_{x\rightarrow r^*_\pm} \hat{u}(x)=0$. Then from $\hat{u}\in L^2(\mathbb{R}^*_0, m^*_0)$ and \eqref{EQROJ}, we can obtain $u\in L^2({I})$. Thus $u\in \mathcal{G}^{(s)}\cap L^2({I})=\mathcal{G}^{(s)}_0$. 

{To the contrary}, take $\hat{u}\in \mathcal{G}^{(s)*}_0$ and let $u$ be given by \eqref{EQUJU}. Since $u\in L^2({I})$, it follows that $\hat{u}\in L^2(\mathbb{R}^*_0,m^*_0)$ and $\lim_{x\rightarrow r^*_\pm}\hat{u}(x)=0$. We only need to prove $\hat{u}$ is absolutely continuous on ${\mathbb{R}^*_0}$ and $\hat{u}'\in L^2({\mathbb{R}^*_0})$. Note that ${\mathbb{R}^*_0}\setminus \{p_n^*:n\geq 1\}$ and $F\setminus H$ have a one-to-one correspondence. Moreover the Lebesgue measure on ${\mathbb{R}^*_0}\setminus \{p_n^*:n\geq 1\}$ corresponds to $1_F(x)dx$ on $F\setminus H$ via the transform $j$.
Since $u'=0$ on $G$, we can define an a.e. defined function $u'\circ j^{-1}$ through darning method. Precisely, $u'\circ j^{-1}(p_n^*):=0$ and $u'\circ j^{-1}(\hat{x}):=u'(j^{-1}(\hat{x}))$ for a.e. $\hat{x}\in  {\mathbb{R}^*_0}\setminus \{p_n^*:n\geq 1\}$ such that $u$ is differentiable at $x:=j^{-1}(\hat{x})$. Clearly $u'\circ j^{-1}\in L^2({\mathbb{R}^*_0})$.
Without loss of generality, take a point $\hat{y}\in {\mathbb{R}^*_0}$ with $\hat{y}>0$. Let $y:=j^{-1}(\hat{y})$ if $\hat{y}\notin  \{p_n^*:n\geq 1\}$ and set $y$ to be the left endpoint of interval $j^{-1}(\hat{y})$ if $\hat{y}\in  \{p_n^*:n\geq 1\}$. Then
\[
	\hat{u}(\hat{y})-\hat{u}(0)=u(y)-u(z)=\int_z^y u'(t)1_F(t)dt=\int_z^y u'\circ j^{-1} (j(t))dj(t)
\]
and it follows that
\[
	\hat{u}(\hat{y})-\hat{u}(0)=\int_{j(z)}^{j(y)}u'\circ j^{-1}(\hat{t})d\hat{t}=\int_0^{\hat{y}}u'\circ j^{-1}(\hat{t})d\hat{t}.
\]
Since $u'\circ j^{-1}\in L^2({\mathbb{R}^*_0})$, we can conclude that $\hat{u}$ is absolutely continuous on ${\mathbb{R}^*_0}$ and
\begin{equation}\label{EQHUUC}
\hat{u}'=u'\circ j^{-1} \text{ a.e. on } {\mathbb{R}^*_0}.
\end{equation}
Therefore $\hat{u}\in H^1_{0,\mathrm{e}}({\mathbb{R}^*_0})\bigcap L^2(\mathbb{R}^*_0,m^*_0)$.

For any $\hat{u},\hat{v}\in \mathcal{G}^{(s)*}_0$, let $u,v$ be given by \eqref{EQUJU} respectively. Note that $u'=v'=0$ on $G$. Thus we have
\[
	\mathcal{E}^*(\hat{u},\hat{v})=\frac{1}{2}\int_\mathbb{R}u'(x)v'(x)dx=\frac{1}{2}\int_\mathbb{R}u'\circ j^{-1}(j(x))v'\circ j^{-1}(j(x))dj(x).
\]
From \eqref{EQHUUC} we can obtain that
\[
	\mathcal{E}^*(\hat{u},\hat{v})=\frac{1}{2}\int_{{\mathbb{R}^*_0}}u'\circ j^{-1}(x)v'\circ j^{-1}(x)dx=\frac{1}{2}\int_{{\mathbb{R}^*_0}} \hat{u}'(x)\hat{v}'(x)dx.
\]
In other words,  \eqref{EQMGS} is proved.

Clearly, $(\mathcal{E}^*,\mathcal{G}^{(s)*}_{0})$ is a regular local Dirichlet form on $L^2(\mathbb{R}^*_0,m^*_0)$, whose associated diffusion process is {an absorbing Brownian motion on $\mathbb{R}^*_0$ being time changed by its positive continuous additive functional with the Revuz measure $m^*_0$ by virtue of Theorem~4.2 of \cite{F10} applied to $s(x)=x$.} 

Finally, we shall prove that the D-space $(\mathbb{R}^*_0, m^*_0,\mathcal{G}^{(s)*}_{0},\mathcal{E}^*)$ is equivalent to {$(I,m_I,\mathcal{G}^{(s)}_0,\mathcal{E})$}. Clearly
the map \begin{align*}
	\Phi:&\ u\mapsto \hat{u},\\
&\mathcal{G}^{(s)}_0\cap L^\infty({I})\rightarrow \mathcal{G}^{(s)*}_{0}\cap L^\infty({\mathbb{R}^*_0}),
\end{align*}
is an algebraic isomorphism, where $\hat{u}$ is given by \eqref{EQUJU}.  From \eqref{EQROJ} and \eqref{EQUJU}, we can conclude that $(u,u)_{m_I}=(\hat{u},\hat{u})_{{m^*_0}}$ and $||u||_\infty=||\hat{u}||_\infty$. Moreover $\mathcal{E}^*(\hat{u},\hat{u})=\mathcal{E}(u,u)$ is direct from the definition of $\mathcal{E}^*$ in \eqref{EQEJUVM}. That completes the proof.
\end{proof}

\subsection{On the traces}

In this section, we return to the question raised at the beginning of \S\ref{OCO}. Similarly to \eqref{EQGSU},  the orthogonal complement of $\check{\mathcal{F}}^{(s)}_\text{e}$ in $\check{\mathcal{F}}_\text{e}$ relative to the quadratic form $\check{\mathcal{E}}$ can be defined as
\begin{equation}\label{EQMGSUF}
	\check{\mathcal{G}}^{(s)}:=\{u\in \check{\mathcal{F}}_\text{e}:\check{\mathcal{E}}(u,v)=0\text{ for any }v\in \check{\mathcal{F}}^{(s)}_\text{e}\}.
\end{equation}
From Theorem~\ref{THM3} and \eqref{EQCMFS}, we can deduce that $$\check{\mathcal{G}}^{(s)}=\mathcal{G}^{(s)}|_F$$ and every $u\in \check{\mathcal{F}}_\text{e}$ can be expressed as
\[
	u=u_1+u_2
\]
for some $u_1\in \check{\mathcal{F}}^{(s)}_\text{e}$ and $u_2\in \check{\mathcal{G}}^{(s)}$. This decomposition is unique if $(\mathcal{E}^{(s)},\mathcal{F}^{(s)})$ is transient and unique up to a constant if $(\mathcal{E}^{(s)},\mathcal{F}^{(s)})$ is recurrent. 

In Theorem~\ref{THM5}, we assert that the regular Dirichlet subspace $(\check{\mathcal{E}}^{(s)},\check{\mathcal{F}}^{(s)})$ contains exactly the information of non-local part
of trace Brownian motion $(\check{\mathcal{E}},\check{\mathcal{F}})$ on $F$. Actually the following lemma shows that $\check{\mathcal{G}}^{(s)}$ only contains the strongly local information of $(\check{\mathcal{E}},\check{\mathcal{F}})$.

\begin{lemma}
	For any $u,v\in \check{\mathcal{G}}^{(s)}$, it holds that
	\[
		\check{\mathcal{E}}(u,v)=\frac{1}{2}\int_F u'(x)v'(x)dx.
	\]
\end{lemma}
\begin{proof}
Note that $\check{\mathcal{G}}^{(s)}=\mathcal{G}^{(s)}|_F$. For any finite component $I_n=(a_n,b_n)$ of $G$, it follows from \eqref{EQMGSU} that $u(a_n)=u(b_n)$ for any $u\in \mathcal{G}^{(s)}$. Hence from Theorem~\ref{THM5} we can complete the proof.  
\end{proof}

Take a smooth measure $\mu$ introduced in \S\ref{TBM} as
\[
	\mu(dx)=1_F(x)dx+\sum_{n\geq 1, d_n<\infty}\frac{d_n}{2}(\delta_{a_n}+\delta_{b_n}),
\]
where $\{I_n=(a_n,b_n):n\geq 1\}$ is the set of components of $G$ and $\delta_p$ is the mass of $p$. Let 
\[
	F_0:= F\bigcap I,\quad \mu_0:=\mu|_{F_0},\quad  \check{\mathcal{G}}^{(s)}_0:= \check{\mathcal{G}}^{(s)}\bigcap L^2(F,\mu), 
\]
where $I=(r_-,r_+)$ is given by \eqref{EQIRR}. Similarly to Lemma~\ref{LM6}, $(\check{\mathcal{E}}, \check{\mathcal{G}}^{(s)}_0)$ is a Dirichlet form on $L^2(F,\mu)$ in the wide sense. A regular representation of $(F, \mu, \check{\mathcal{G}}^{(s)}_0, \check{\mathcal{E}})$ can be achieved by the similar darning method to \S\ref{SEC32}. By the darning transform $j|_F$, we may  similarly define $F_0^*, \mu_0^*, \check{\mathcal{G}}^{(s)*}_0$ and $\check{\mathcal{E}}^*$. After a little computation, we have
\[
	F^*_0=\mathbb{R}^*_0, \ \mu^*_0=m^*_0, \ \check{\mathcal{G}}^{(s)*}_0=\mathcal{G}^{(s)*}_0,\ \check{\mathcal{E}}^*={\mathcal{E}}^*. 
\]
Hence we have the following theorem by Theorem~\ref{prop5}. 

\begin{theorem}\label{THM6}
The quadratic form $(\check{\mathcal{E}}^*, \check{\mathcal{G}}_{0}^{(s)*})$ defined above is a regular strongly local Dirichlet form on $L^2(F^*_0, \mu^*_0)$, whose associated diffusion process is {a time-changed absorbing Brownian motion} on $F^*_0$ with the speed measure $\mu^*_0$.
Furthermore, the D-spaces {$(I, m_I,\mathcal{G}^{(s)}_0,\mathcal{E})$} and $(F,\mu, \check{\mathcal{G}}^{(s)}_0,\check{\mathcal{E}})$ are equivalent, and
$$(\mathbb{R}^*_0,m^*_0,\mathcal{G}^{(s)*}_{0}, \mathcal{E}^*)=(F^*_0, \mu^*_0, \check{\mathcal{G}}_{0}^{(s)*},\check{\mathcal{E}}^*)$$  is their common regular representation.
\end{theorem}

In a word, the trace Brownian motion on $F$ may be decomposed as a regular Dirichlet subspace which contains its non-local part and the orthogonal complement which contains its local part and has a regular representation.
The {time-changed absorbing Brownian motion} in Theorem~\ref{THM6} is called the orthogonal darning process of $\check{\mathcal{F}}^{(s)}_\text{e}$ relative to $\mu$. If we replace $\mu$ with another Radon measure $\mu'$ on $\mathbb{R}$ with support $F$, it is similar to obtain that the orthogonal darning process of $\check{\mathcal{F}}^{(s)}_\text{e}$ relative to $\mu'$ is an absorbing time-changed Brownian motion on $\mathbb{R}^*_0$ whose speed measure is $\mu'$. In particular, if $\mu'(dx)=1_F(x)dx$, then it is actually an absorbing Brownian motion on $\mathbb{R}^*_0$. Note that they are equivalent up to a time-change transform.

{
\section*{Acknowledgement}
The authors would like to thank an anonymous reviewer for his helpful suggestions to improve our paper. 
}

\end{document}